\documentclass[12pt,a4paper]{article}
\usepackage[nottoc,numbib]{tocbibind}
\usepackage[english]{babel}
\usepackage[utf8x]{inputenc}
\usepackage{amsfonts,longtable,amssymb,amsmath,latexsym,dsfont}
\usepackage{wrapfig}
\usepackage{amsthm, amsmath, amssymb, amsfonts, amscd, amscd, geometry}
\newsavebox{\ssa}
\usepackage{titling}
\usepackage{authblk}

\newtheorem{thm}{Theorem}[subsection]
\newtheorem{prop}[thm]{Proposition}
\newtheorem{lemma}[thm]{Lemma}
\newtheorem{cor}[thm]{Corollary}

\theoremstyle{definition}
\newtheorem{def0}[thm]{Definition}

\theoremstyle{remark}
\newtheorem*{rem}{Remark}
\newtheorem{ex}[thm]{Example}

\newcommand{\lthm}{Łoś's theorem }
\newcommand{\Rep}{\textrm{Rep}}
\newcommand{\Hom}{\textrm{Hom}}
\newcommand{\End}{\textrm{End}}
\newcommand{\Id}{\textrm{Id}}
\newcommand{\Irr}{\textrm{Irr}}

\newcommand{\Aut}{\textrm{Aut}}
\long\def\/*#1*/{}

\title{Finite-dimensional representations of Yangians in complex rank}
\author{Daniil Kalinov\thanks{dkalinov@mit.edu}}
\affil{\footnotesize Department of Mathematics, Massachusetts Institute of Technology, Cambridge, MA 02139, USA}
\affil{\footnotesize Departament of Mathematics, Higher School of Economics, 119048 Moscow, Russia}
\affil{\footnotesize Landau Institute for Theoretical Physics, Chernogolovka, 142432 Moscow Region, Russia}

\date{}

\begin{document}

\maketitle

\begin{abstract}
We classify the "finite-dimensional" irreducible representations of the Yangians $Y(\mathfrak{gl}_t)$ and $Y(\mathfrak{sl}_t)$. These are associative ind-algebras in the Deligne category $\Rep(GL_t)$, which generalize the regular Yangians $Y(\mathfrak{gl}_n)$ and $Y(\mathfrak{sl}_n)$ to complex rank. They were first defined in the paper \cite{etingof2016representation}. Here we solve Problem 7.2 from \cite{etingof2016representation}. We work with the Deligne category $\Rep(GL_t)$ using the ultraproduct approach introduced in  \cite{deligne2007categorie} and \cite{harman2016deligne}. 
\end{abstract}

\tableofcontents

\section*{Introduction}

The study of representation theory in complex rank was initiated by Pierre Deligne. In his paper \cite{deligne2007categorie} he defined the tensor categories $\Rep(S_t)$, $\Rep(GL_t)$, $\Rep(SO_t)$ and $\Rep(Sp_t)$ interpolating the categories of representations of the corresponding groups in finite rank\footnote{Actually the category $\Rep(GL_t)$ appears even earlier in the paper \cite{deligne1982tannakian} by Deligne and Milne and \cite{deligne1990categories} by Deligne}. These categories were studied by Deligne himself and many other authors, for example see \cite{comes2012ideals}, \cite{comes2009blocks}, \cite{knop2006construction}.

Later Etingof in his papers \cite{etingof2014representation}, \cite{etingof2016representation} suggested the methods for interpolating the representation theory of many other algebras connected to $S_n$ or $GL_n$ to complex rank. These papers included many open problems, some of which were  solved, for example see \cite{entova2014schur},\cite{entova2014representations},\cite{harman2016generators} and \cite{sciarappa2015simple}. 

In this paper we study the representation theory of the generalization of the Yangians $Y(\mathfrak{gl}_n)$ and $Y(\mathfrak{sl}_n)$ to complex rank, which were introduced by Etingof in section 7 of \cite{etingof2016representation}, thus solving Problem 7.2 in \cite{etingof2016representation} in the case of general linear and special linear groups. 

The Yangian $Y(\mathfrak{g})$ was introduced by Vladimir Drinfeld in his paper \cite{drinfeld1985hopf} for a general simple Lie algebra $\mathfrak{g}$. He later classified the finite-dimensional irreducible representations of $Y(\mathfrak{g})$ in \cite{drinfeld1988new}. The theory of Yangians for $\mathfrak{gl}_n$ and $\mathfrak{sl}_n$ is described in detail also in the textbooks \cite{molev2007yangians}, \cite{chari1995guide}. These books use somewhat different approaches, the first one highly relying on the Faddeev-Reshetikhin-Takhtajan presentation of $Y(\mathfrak{gl}_n)$. We will mostly follow the notations and techniques of \cite{molev2007yangians}, but we will describe some connections with methods of \cite{chari1995guide} later on. The main result of this paper is the classification of "finite-dimensional" irreducible representations of $Y(\mathfrak{sl}_t)$ and $Y(\mathfrak{gl}_t)$ by Drinfeld polynomials, which generalizes the known theorems in the finite rank case. 

We can not prove this theorem using the same techniques as in the finite rank case, since there is no obvious way to generalize the triangular decomposition of $Y(\mathfrak{gl}_n)$ or $Y(\mathfrak{sl}_n)$ to the complex rank case. Thus we need to use some instruments to connect results in finite rank case to the result in complex rank.

There are several methods of generalizing results from the representation theory in finite rank cases to the representation theory in Deligne categories. In this paper we will use the method of ultraproducts, which was introduced for transcendental rank in \cite{deligne2007categorie} and later generalized to algebraic ranks in the work \cite{harman2016deligne} motivated by Deligne's letter to Ostrik. This method provides a way to generalize results from finite rank to both algebraic and transcendental $t$ provided we know enough about the representation theory in positive characteristic. For other methods see for example \cite{mathew2013categories}.

We also would like to note that previously this subject was studied by Léa Bittmann in \cite{bittmann}. There she has studied the representation of Yangians $Y(\mathfrak{sl}_t)$ and $Y(\mathfrak{gl}_t)$ using the invariant algebra $Y(\mathfrak{g})^{\mathfrak{g}}$.

The sturcture of the paper is as follows. In section 1, we first briefly discuss the construction of the Deligne category $\Rep(GL_t)$ and its properties which are going to be relevant to us. Later in the section we review the theory of ultrafilters and ultraproducts and describe the contruction of $\Rep(GL_t)$ as an ultraproduct of the categories of representations of $GL_n$.

In section 2 we recall the Faddeev-Reshetikhin-Takhtajan presentation of $Y(\mathfrak{gl}_t)$ and describe its basic properties, define the subalgebra $Y(\mathfrak{sl}_t)$, and describe the classification theorem of irreducible finite-dimensional representations of $Y(\mathfrak{gl}_n)$ and $Y(\mathfrak{sl}_n)$.

In section 3 we generalize some of the results of section 2 to positive characteristic under some conditions on $n$ and $p$. These results will be crucial for us when we work with $Y(\mathfrak{gl}_t)$ for algebraic $t$.

In section 4 we first define $Y(\mathfrak{gl}_t)$ and $Y(\mathfrak{sl}_t)$ and prove some of their properties, using the ultraproduct approach. Later we study irreducible representations of these algebras and prove the classification theorem for them.

\subsection*{Acknowledgement} 

I want to thank Pavel Etingof for suggesting this project to me  and for the valuable discussions we had about it. 

The work has been supported by the RScF grant 16-12-10151.

\section{Deligne category $\Rep (GL_t)$}

\subsection{Construction of $\Rep (GL_t)$ and its main properties.}

Here we will give a definition of the category $\Rep(GL_t)$ for $t \in \mathbb C$ and state a number of important properties of this category. For a more detailed discussion, see \cite{etingof2016representation}, \cite{comes2012ideals}.

First, we define a preliminary category $\Rep_0(GL_t)$:
\begin{def0}
$\Rep_0(GL_t)$ is a skeletal symmetric tensor category with objects being given by pairs $(n,m)$ with $n,m \in \mathbb Z_{\ge 0}$, represented as rows consisting of $n$ $\bullet$'s and $m$ $\circ$'s.

$\Hom_{\Rep_0(GL_t)}((m,n),(m',n'))$ is a vector space over $\mathbb C$ with a basis given by all possible matchings, i.e. ways to pair all elements of two rows corresponding to $(n,m)$ and $(n',m')$ such that $\bullet$ can be paired with $\bullet$ or $\circ$ with $\circ$ only if they belong to different rows, and $\bullet$ can be paired with $\circ$ only if they are in the same row.

Composition is given by vertical concatenation of diagrams and forgetting the middle row, and then deletion of each resulting loop and multiplying the coefficient of the corresponding basis vector by $t$ for each loop deleted. Tensor product is given by horizontal concatenation of diagrams.
\end{def0}

Now we can easily construct the Deligne category from this:
\begin{def0}
The Deligne category $\Rep(GL_t)$ is the Karoubian envelope of the additive envelope of $\Rep_0(GL_t)$.
\end{def0}

We will also need auxiliary definitions:
\begin{def0}
A bipartition $\lambda$ is a pair of partitions $\lambda = (\lambda^\bullet, \lambda^\circ)$. The size of $\lambda$ is equal to $|\lambda| = |\lambda^\bullet| + |\lambda^\circ|$. The length of $\lambda$ is equal to $l(\lambda) = l(\lambda^\bullet) + l(\lambda^\circ).$
\end{def0}

\begin{def0}
An object $V((\square, \emptyset)) = (1,0)$ is called the fundamental representation and is denoted by $V$. An object $V((\emptyset,\emptyset)) = (0,0)$ is called the trivial representation and is denoted by $\mathbb C$ (abusing notation).
\end{def0}

Here are some properties of $\Rep(GL_t)$:
\begin{prop}
\textbf{a)} The category $\Rep(GL_t)$ is semisimple for $t\notin \mathbb Z$.

\textbf{b)} The irreducible objects of $\Rep(GL_t)$ are in $1-1$ correspondence with bipartitions of arbitrary size. They are denoted by $V(\lambda)$. Moreover $V(\lambda)$ is a direct summand in $(r,s)$.

\textbf{c)} The dimension of $V$ is $t$ and the dimension of $\mathbb C$ is $1$.

\textbf{d)} $V((\lambda^\bullet,\lambda^\circ))^* = V((\lambda^\circ, \lambda^\bullet))$
\end{prop}

Also $\Rep(GL_t)$ has an important universal property. Suppose $\mathcal T$ is a symmetric rigid tensor category and denote by $\mathcal T_t$ the full subcategory of $t$-dimensional objects in $\mathcal T$. Also by $\mathcal{H}om(\Rep(GL_t),\mathcal T)$ denote a category of symmetric tensor functors between $\Rep(GL_t)$ and $\mathcal T$. When we have the following proposition (Thm. 2.9(ii) in \cite{etingof2016representation} and Prop. 10.3 in \cite{deligne2007categorie}):
\begin{prop}
\textbf{(Deligne)}The following functor induces an equivalence of categories:
\begin{align*}
    \mathcal{H}om(\Rep(GL_t),\mathcal T) &\rightarrow \mathcal T_t \\
    F &\mapsto F(V) \\ 
    (\eta: F \Rightarrow F') &\mapsto \eta_V \ .
\end{align*}
\end{prop}
In particular this means that for any $t$-dimensional object in $\mathcal{T}$ we have a symmetric tensor functor $\Rep(GL_t) \to \mathcal{T}$ sending $V$ to this object.

\subsection{Ultrafilters and ultraproducts.}

It is crucial to us that there is another construction of $\Rep(GL_t)$ as a subcategory in a certain ultraproduct category, which formalizes the fact that $\Rep(GL_t)$ is a "limit" of representation categories of general linear groups.

We will quickly define what ultrafilters and ultraproducts are, state their main properties and give some examples. For more details see \cite{schoutens2010use}.

\begin{def0}
An ultrafilter $\mathcal F$ on a set $X$ is a subset of $2^{X}$ satisfying the following properties:

$\bullet$ $X \in \mathcal F$ ;

$\bullet$ If $A \in \mathcal F$ and $A \subset B$, then $B \in \mathcal F$ ;

$\bullet$ If $A,B \in \mathcal F$, then $A\cap B \in \mathcal F$ ;

$\bullet$ For any $A\subset X$ either $A$ or $X \backslash A$ belongs to $A$, but not both.
\end{def0}

There is an obvious family of examples of ultrafilters: $\mathcal F_x = \{ A| x \in A \}$ for $x \in X$. Such ultrafilters are called principal. Using Zorn's lemma one can show that there exist  non-principal ultrafilters $\mathcal F$. Also it follows that all cofinite sets belong to such an $\mathcal F$ (but not all  sets belonging to $\mathcal F$ are cofinite). From now on we will denote by $\mathcal F$ a fixed non-principal ultrafilter on $\mathbb N$. Also by something being true for "almost all $n$", we will mean that it is true for all $n$ in some $A \in \mathcal F$. Note that by definition of an ultrafilter, if two statements hold for almost all $n$, then their conjunction holds for almost all $n$. Also note that if  for almost all $n$ the disjunction of a finite number of statements holds, then one of them holds for almost all $n$ (if not then each of them holds on some subset $A \notin \mathcal F$ and the union of this subsets is not in $\mathcal F$). We will use these elementary observations quite frequently. 

Let's now define a notion of an ultraproduct. 
\begin{def0}
Suppose we have a collection of sets $S_i$ labeled by natural numbers. Suppose that for almost all $x\in A$ one has $S_x \ne \emptyset$. Then $\prod_{\mathcal F}S_x$ is the quotient of $\prod_{x \in A} S_x$ by the following relation: $\{s_x\} \sim \{s'_x\}$ iff $s_x = s_x'$ for almost all $x$.
If for almost all $x$ one has $S_x = \emptyset$, then $\prod_{\mathcal F}S_x = \emptyset$. 
The set $\prod_{\mathcal F}S_x$ is called the ultraproduct of $S_x$. 
\end{def0}
Usually we will denote $\{ s_x \} \in \prod_{\mathcal F}S_i$ by $\prod_\mathcal F s_x$.

The most important property of ultraproducts is the following:
\begin{thm}\textbf{\lthm} (Thm. 2.3.2 in \cite{schoutens2010use})

Suppose we have a collection of sequences of sets $S^{(k)}_i$ for $k = 1,\dots,m$ and a collection of sequences of elements $f^{(r)}_i$ for $r = 1,\dots, l$ and a  formula of a first order language $\phi(x_1,\dots,x_l, Y_1, \dots, Y_m)$ depending on some parameters $x_i$ and sets $Y_j$. Denote by $S^{(k)} = \prod_{\mathcal F}S^{(k)}_{n}$ and $f^{(r)} = \prod_{\mathcal F} f^{(r)}_n$.  Then 
$\phi(f^{(1)}_n, \dots, f^{(l)}_n, S^{(1)}_n, \dots S^{(m)}_n)$ is true for almost all $n$ iff  $\phi(f^{(1)}, \dots, f^{(l)}, S^{(1)}, \dots S^{(m)})$ is true.
\end{thm}

In plain language this means that if we have a sequence  of collections of sets with some algebraic structure given by maps between them, then, first, we have the corresponding maps between the ultraproducts of these sets. And, second, these maps satisfy a given set of axioms or properties for the ultraproducts iff they satisfy these axioms/properties for almost all $n$. Also frequently it is useful to think about  ultraproducts as a some kind of limits as $n \mapsto \infty$.

We  give a number of examples of such constructions, which are going to be useful to us below:

\begin{ex}
 If $S_i$ is a sequence of monoids/groups/rings/fields then $\prod_{\mathcal F} S_i$ with operations given by taking the ultraproduct of the operations in the corresponding sets of $\Hom_{Sets}$ gives us a structure of monoid/group/ring/field by Łoś's theorem.
\end{ex}

\begin{ex}
 If $V_i$ are finite-dimensional vector spaces over a field $k$, then $\prod_{\mathcal F} V_i$ is not necessarily a finite-dimensional vector space, since the property of being finite-dimensional cannot be written in first-order language. But if the dimensions of $V_i$ are bounded, then they are the same for almost all $i$ and hence $V$ has the same dimension (for example, because the ultraproduct of bases is a basis).
\end{ex}

\begin{ex}
 Take the ultraproduct of a countably infinite number of copies of $\overline{\mathbb Q}$. By \lthm $\prod_{\mathcal F} \overline{\mathbb Q}$ is a field, which is algebraically closed. It has characteristic zero since $\forall k  \in \mathbb Z$ such that $k\ne 0$ it follows that $ k = \prod_{\mathcal F} k\ne 0$.  Also it is easy to see that its cardinality is continuum. Hence by Steinitz's theorem\footnote{This theorem tells us that two uncountable algebraically closed fields are isomorphic iff their characteristic and cardinality are the same. It is proven in \cite{steinitz1910algebraische}.} $\prod_{\mathcal F} \overline{\mathbb Q} \simeq \mathbb C$. Note that there is no canonical isomorphism.
\end{ex}
 
 \begin{ex}
 Take the ultraproduct of $\overline{\mathbb F}_{p_n}$ for some sequence of distinct prime numbers $p_n$. As before, by \lthm $\prod_{\mathcal F} \overline{\mathbb F}_{p_n}$ is a field, which is algebraically closed. Also as before it has cardinality continuum. Now $k = \prod_{\mathcal F} k \ne 0$, since it is equal to zero for at most one $k$.  Hence $\prod_{\mathcal F} \overline{\mathbb F}_{p_n} \simeq \mathbb C$, again not in a canonical way.
\end{ex}

\begin{ex}
 Suppose $\mathcal C_i$ is a collection of small categories. We can define an ultraproduct category $\widehat{\mathcal C} = \prod_{\mathcal F} \mathcal C_i$ as a category with objects $Ob( \widehat{ \mathcal C}) = \prod_{\mathcal F} Ob(\mathcal C_i)$ and $\Hom_{\widehat{\mathcal C}}(\prod_{\mathcal F} X_i,\prod_{\mathcal F} Y_i) = \prod_{\mathcal F} \Hom_{\mathcal C_i}(X_i,Y_i)$; the composition maps are given by the ultraproducts of composition maps, i.e. $(\prod_{\mathcal F}f_i) \circ (\prod_{\mathcal F}g_i) = \prod_{\mathcal F} (f_i \circ g_i)$. By \lthm this data satisfies the axioms of a category. If the categories $\mathcal C_i$ have some structures, for example the structures of an abelian/monoidal/tensor category, then $\widehat{\mathcal C}$ also has these structures\footnote{But the finite-length property, for example, does not survive}.
 
 Usually $\widehat{\mathcal C}$ is too big and it is interesting to consider some full subcategories $\mathcal C$ in there, or, equivalently, consider ultraproducts only of some sequences of objects of $\mathcal C_i$, for example bounded in some sense. 
 
 This construction obviously extends to essentially small categories. All categories which we will consider are essentially small, so we won't bother mentioning this later.
\end{ex}

\subsection{Deligne categories as ultraproducts.}

Here, we will  show how to construct $\Rep(GL_t)$ using ultraproducts. See \cite{harman2016deligne}, \cite{deligne2007categorie}.

We want to apply the last example of the previous section to $\mathcal C_i = \textbf{Rep}(GL_{n_i}, \mathbb K_i)$ -- the tensor category of finite-dimensional representations of $GL_{n_i}$ over $\mathbb K_i$. From now on we will denote by $\textbf{Rep}(GL_n) = \textbf{Rep}(GL_n,\overline{\mathbb Q})$ and by $\textbf{Rep}_p(GL_n) = \textbf{Rep}(GL_n, \overline{\mathbb F}_p)$. We have the following result (Introduction of \cite{deligne2007categorie} or Thm. 1.1 in \cite{harman2016deligne}):
\begin{thm}

\textbf{a)} Suppose $t\in \mathbb C$ is transcendental. Consider $\widehat{\mathcal C} = \prod_{\mathcal{F}} \textbf{Rep}(GL_n)$. Denote by $V_i$ -- the fundamental representation of $GL_i$ and $V = \prod_{\mathcal{F}}V_i$. Fix an isomorphism $\prod_{\mathcal F}\overline{\mathbb Q}\simeq \mathbb C$ such that $\prod_{\mathcal F} i = t$. Then the full subcategory of the $\prod_{\mathcal F}\overline{\mathbb Q}$-linear category $\widehat{\mathcal C}$ generated by $V$ under taking duals, tensor products, direct sums and direct summands is equivalent to the $\mathbb C$--linear category $\Rep(GL_t)$, in a way consistent with the above isomorphism $\prod_{\mathcal F}\overline{\mathbb Q} \simeq \mathbb C$.

\textbf{b)} Suppose $t \in \mathbb C$ is algebraic but not integer, with minimal polynomial $q(x) \in \mathbb Z[x]$. Fix a sequence of distinct primes $p_n$ and sequence of integers $t_n$ tending to infinity such that $q(t_n) = 0$ in $\mathbb F_{p_n}$. Moreover fix an isomorphism $\prod_{\mathcal F}\overline{\mathbb F}_{p_n}\simeq \mathbb C$ such that $\prod_{\mathcal F} t_i = t$. Set $\widehat{\mathcal C} = \prod_{\mathcal{F}} \textbf{Rep}_{p_n}(GL_{t_n})$. Denote by $V_{t_i}$  the fundamental representation of $GL_{t_i}$ and set $V_t = \prod_{\mathcal{F}}V_{t_i}$. Then the  full subcategory of the $\prod_{\mathcal F}\overline{\mathbb F}_{p_n}$-linear category $\widehat{\mathcal C}$ generated by $V$ under taking duals, tensor products, direct sums and direct summands is equivalent to the $\mathbb C$-linear category $\Rep(GL_t)$,  in a way consistent with the above isomorphism $\prod_{\mathcal F}\overline{\mathbb F}_{p_n}\simeq \mathbb C$.
\end{thm}
\begin{proof}
\textbf{a)} First let's prove that it is indeed possible to fix such an isomorphism. The ultraproduct $\prod_{\mathcal F}i$ is an element of $\mathbb C$. Suppose it is algebraic over $\mathbb Q$, then it should satisfy a monic equation $f$ with coefficients in $\mathbb Q$. Then by \lthm for almost all $i$ we have $f(i) = 0$, but since this is true for infinite number of distinct $i$'s, it follows that $f = 0$. Hence by contradiction we conclude that $\prod_{\mathcal F} i$ is transcendental. Now by fixing an automorphism of $\mathbb C$ over $\mathbb Q$ we may send this transcendental number to $t$.

So we have a tensor category $\widehat{\mathcal C}$ linear over $\mathbb C$, with an object $\prod_{\mathcal F}V_i$ of dimension $t$. Hence by Prop. 1.1.6 we obtain a tensor functor $F:\Rep(GL_t) \to \widehat{\mathcal C}$. Since $\Rep(GL_t)$ is generated by $V$ under taking duals, tensor products, direct sums and direct summands, it follows that the image of $\Rep(GL_t)$ under $F$ is contained in the full subcategory $\mathcal C$ in $\widehat{\mathcal C}$ generated by $V_t$ under taking duals, tensor products, direct sums and direct summands. So we know that $F:\Rep(GL_t) \to \mathcal C$ is essentialy surjective. Now it is enough to prove that it is fully faithful. 

Note that it is enough to prove that 
$$
\prod_{\mathcal F} \Hom_{GL_n}(V_n^{\otimes r} \otimes V_n^{* \otimes s}, V_n^{\otimes p} \otimes V_n^{* \otimes q}) = \Hom_{\mathcal C}(V_t^{\otimes r} \otimes V_t^{* \otimes s}, V_t^{\otimes p} \otimes V_t^{* \otimes q}) \ ,
$$
and that the composition maps are the same. Indeed both categories can be obtained as the Karoubian envelope of the additive envelope of the categories consisting of all $V^{\otimes p} \otimes V^{* \otimes q}$ or  $V_t^{\otimes p} \otimes V_t^{* \otimes q}$ respectively.

From Schur-Weyl duality we know that for a postive integer $r$, another positive integer $n>r$, the algebra $\End_{GL_n}(V_n^{\otimes r}) = (V_n^{\otimes r} \otimes V_n^{*\otimes r})^{GL_n}$ is naturally isomorphic to $\overline{\mathbb Q}[S_r]$. From invariant theory we also know that $(V_n^{\otimes r} \otimes V_n^{*\otimes r})^{GL_n}$ is generated by the elements $\sum_{i_1,\dots,i_r} e_{i_1}\otimes \dots \otimes e_{i_r} \otimes \varepsilon_{i_{\sigma(1)}}\otimes \dots \otimes \varepsilon_{i_{\sigma(r)}}$, where $e_i$ and $\varepsilon_j$ are dual bases of $V_n$ and $V_n^*$ and $\sigma \in S_r$.

Since $\Hom_{GL_n}(V_n^{\otimes r} \otimes V_n^{* \otimes s}, V_n^{\otimes p} \otimes V_n^{* \otimes q}) = \left(V_n^{\otimes(r+q)}\otimes V_n^{*\otimes(s+p)}\right)^{GL_n}$, it follows that it is non-zero only if $r+q=p+s$ and for sufficently large $n$ has dimension $(r+q)!$ . Now rewriting the generating elements above using evaluation and coevaluation maps, we may explicitly describe these elements. It is easy to see that these descriptions can be obtained by looking at the diagram of a fixed matching and interpreting lines connecting different rows as identity arrows, lines connecting elements in the upper row as evaluation maps and lines connecting elements in the lower row as co-evaluation maps. Since the number of matchings is $(r+q)!$, it follows that for sufficiently large $n$ they form a basis. Since they form the basis of $\Hom_{\Rep(GL_t)}(V^{\otimes r} \otimes V^{* \otimes s}, V^{\otimes p} \otimes V^{* \otimes q})$ by definition, the equality of $\Hom$-spaces follows (since taking ultraproduct of the same finite-dimensional vector space over $\overline{\mathbb Q}$ gives you the same vector space tensored with $\mathbb C$). 

It remains to check that the composition maps are the same. In the  $GL_n$ case they are given also by vertical concatenation of diagrams. The only thing which we need to check is what happens to the loops obtained in this way. But since each loop through the properties of tensor categories can be simplified to $coev \circ ev$, it follows that each loop gives us a multiplication by $n$, hence for the composition maps in the ultraproduct we get the multiplication by $t$, so the composition law in $\mathcal{C}$ and $\Rep(GL_t)$ is the same and we are done.

\textbf{b)} First, again, we need to explain how we can fix such an isomorphism.
Let us prove that there is indeed an infinite number of pairs $t_n$ and $p_n$ such that $q(t_n) = 0 \mod p_n$. It is enough to show that there are infinite number of primes dividing the numbers $q(n)$ (if in this case the sequence $t_n$ is bounded, it follows that some $q(t_n)$ is divisible by an infinite number of prime numbers, which is absurd). Suppose it is not so, and there are only $k$ such primes. Fix $C$ such that  we have $q(n) < C \cdot n^{\deg(q)}$ for all positive  integer $n$. Denote by $Q$ the number of integers of the form $q(n)$ for $n \in \mathbb Z_{\ge 0}$ such that $q(n)<N$. By the above inequality this number is at least $\frac{1}{C}\cdot N^{\frac{1}{\deg(q)}}$. On the other hand the number $P$ of  numbers less than $N$ divisible only by $k$ fixed primes  is less or equal to $\log_2(N)^k$, since each prime number is at least $2$. Hence for big enough $N$ we have $P<Q$, which contradicts the hypothesis\footnote{This proof is also written by Nate Harman in his paper, see the proof of Prop. 2.2 in \cite{harman2016deligne}}. 

So we indeed can choose such unbounded sequences $t_n$ and $p_n$. Now by \lthm it follows that $\prod_{\mathcal F}t_n$ is a root of $q$ in $\mathbb C$, so by composing with an automorphism of $\mathbb C$ we may assume that under an isomorphism $\mathbb C \simeq \prod_{\mathcal F} \overline{\mathbb F}_{p_n}$, $\prod_{\mathcal F}t_n$ maps to $t$. 

The remaining part of the proof is completely the same since the relevant part of Schur-Weyl duality holds over an algebraically closed field of any characteristic. Indeed the natural isomorphism $k[S_r] = End_{GL_n}((k^n)^{\otimes r})$ for $n>r$ holds over an infinite field $k$ of any characteristic, see Thm. 1.2 in \cite{konig2001double}.

\end{proof}

Through this isomorphism we obtain a connection between irreducible objects in $\textbf{Rep}(GL_n)$ or $\textbf{Rep}_p(GL_n)$ and irreducible objects of $\Rep(GL_t)$. As a first ingredient for this we need to relate bipartitions with partitions.

\begin{def0}
For a bipartition $\lambda$ denote by $\lambda|_n$ the weight equal to  
{\footnotesize $(\lambda|_n)_i = \lambda^\bullet_i - \lambda^\circ_{n-i+1}$}. Also for a complex number $c\in \mathbb C$ denote by $c|_n$ elements of $\overline{\mathbb Q}$ or $\overline{\mathbb F}_{p_n}$ (which one will be clear in the context) such that $\prod_{\mathcal F}c|_n = c$.
\end{def0}

To show how this works let's describe the sequence of objects of $\textbf{Rep}(GL_n)$ corresponding to $V(\lambda)$. Fix $\lambda$; as we know, $V(\lambda)$ is a direct summand in $ V^{\otimes |\lambda^\bullet|} \otimes V^{*\otimes |\lambda^\circ|}$. Denote by $B_{r,s}(t) = \End(V^{\otimes r} \otimes V^{*\otimes s})$  the walled Brauer algebra. For $n$  big enough $B_{r,s}(n) = \End(V_n^{\otimes r} \otimes V_n^{*\otimes s})$ by \lthm and proof of Thm. 1.3.1. So we can take the direct summand corresponding to the same idempotent as $\lambda$. The corresponding representation is going to be isomorphic to $V(\lambda|_{n})$. So it follows that $\prod_{\mathcal F}V(\lambda|_{n}) = V(\lambda)$.

The same is true for finite characteristic case since, as we show below, for $n$ big enough $V(\lambda|_{t_n})$ is going to be irreducible.

\section{Yangians in integer rank}

In the first two parts of this section we will work over an algebraically closed field of characteristic $0$, more specifically $\overline{\mathbb Q}$ or $\mathbb C$. Here we will briefly recall the definition and basic properties of the Yangian, and then state the classification theorems for finite-dimensional representations of Yangians for general and special linear groups. For a more detailed study see \cite{molev2007yangians} or \cite{chari1995guide} (we will mostly follow \cite{molev2007yangians}).

\subsection{The Yangians $Y(\mathfrak{gl}_n)$ and $Y(\mathfrak {sl}_n)$.}

\begin{def0} (Sections 1.1-1.2 in \cite{molev2007yangians})
$Y(\mathfrak{gl}_n)$ is the associative algebra generated by $t^{(k)}_{ij}$, where $1 \le i,j \le n$ and $k > 0$ with the following defining relation in \\ $Y(\mathfrak gl_n) \otimes \End( \mathbb C^n) \otimes \End (\mathbb C^n)[[u^{-1},v^{-1}]]$:
\begin{equation}
    (u-v)R(u-v)T^I(u)T^{II}(v) = (u-v)T^{II}(v)T^I(u)R(u-v) \ ,
\end{equation}
where $T^{\alpha}(u) = 1 + \sum_{k>0, i,j} t^{(k)}_{ij}\otimes e^{\alpha}_{ij} \ u^{-k} $ and $R(u) = 1 +\frac{\sigma}{u}$ with $\sigma$ being the operator permuting first and second copy of $\mathbb C^n$.
\end{def0}

Below we list properties of this algebra which are going to be important for us:
\begin{prop} (Sections 1.1-1.7 in \cite{molev2007yangians})

\textbf{a)} $Y(\mathfrak{gl}_n)$ is a Hopf algebra with $S(T(u)) = T^{-1}(u)$ and $\Delta(T(u)) = T^I(u)T^{II}(u)$.

\textbf{b)} There is an algebra homomorphism $ev:Y(\mathfrak{gl}_n)\rightarrow U(\mathfrak{gl}_n)$ given by $T(u)\mapsto R(u)$. 

\textbf{c)} There is a Hopf algebra embedding $i:U(\mathfrak{gl}_n) \rightarrow Y(\mathfrak{gl}_n)$ given by $E_{ij} \mapsto t^{(1)}_{ij}$.

\textbf{d)} The center $Z(Y(\mathfrak{gl}_n))$ is generated by algebraically independent elements given by the coefficients of a certain series $\textrm{qdet} \ T(u)$ defined as:
$$
\textrm{qdet} \ T(u) = \sum_{s \in S_n} sgn(s) \cdot t_{1,s(1)}(u-n+1)\dots t_{n,s(n)}(u) \ .
$$

\textbf{e)} There is a family of automorphisms of $Y(\mathfrak{gl}_n)$ given by $T(u) \mapsto f(u)T(u)$ for any $f(u) \in 1 +u^{-1}\mathbb C[[u^{-1}]]$.
\end{prop}

From Prop. 2.1.2(b) it follows that any $\mathfrak{gl}_n$ representation has a structure of a $Y(\mathfrak{gl}_n)$ representation, such representations are called evaluation modules. Also from Prop. 2.1.2(c) it follows conversely that any $Y(\mathfrak{gl}_n)$ representation can be regarded as a $\mathfrak{gl}_n$ representation with respect to the $t^{(1)}_{ij}$-action.

Now we can define the Yangian of the special linear group following Section 1.8 in \cite{molev2007yangians}.
\begin{def0}
$Y(\mathfrak{sl}_n)$ is the sublagebra of $Y(\mathfrak{gl}_n)$ of elements invariant under all automorphisms specified in Prop. 2.1.2(e).
\end{def0}

Below are the properties of this algebra:
\begin{prop}

\textbf{a)} $Y(\mathfrak{sl}_n)$ is a Hopf subalgebra in $Y(\mathfrak{gl}_n)$

\textbf{b)} There is a Hopf algebra embedding $i:U(\mathfrak{sl}_n) \rightarrow Y(\mathfrak{sl}_n)$ given by the restriction of $i$ from Prop.  2.1.2c).

\textbf{c)} $Y(\mathfrak{gl}_n) = Y(\mathfrak{sl}_n) \otimes Z(Y(\mathfrak{gl}_n))$ and thus  $Y(\mathfrak{sl}_n) = Y(\mathfrak{gl}_n)/(\textrm{qdet} \ T(u)-1)$.
\end{prop}

\subsection{Finite-dimensional representation of $Y(\mathfrak{gl}_n)$ and $Y(\mathfrak {sl}_n)$.}

One can classify all finite-dimensional irreducible representations of the Yangians defined above. This can be done by using an analog of the Cartan decomposition for the Yangian, introduction of  an analog of category $\mathcal O$ and by reduction to the $Y(\mathfrak{gl}_2)$ case which can be solved more or less explicitly. More precisely, one can show that each finite-dimensional module has a highest weight vector with a weight given by $(\lambda_1(u),\dots, \lambda_n(u))$, which determine the action of $t_{ii}^{(n)}$ on this vector.
The result of this is the classification theorem originally due to Drinfeld. To state it we need the following definition:
\begin{def0}
Evaluation modules of $Y(\mathfrak{gl}_n)$ or $Y(\mathfrak{sl}_n)$ are irreducible representations of $\mathfrak{gl}_n$ with a structure of the representation of the Yangian given by the morphism $ev:Y(\mathfrak{gl}_n)\to U(\mathfrak{gl}_n)$.
\end{def0}

\begin{rem}
  Note that there are slight differences in the definitions of evaluation modules in \cite{chari1995guide} and \cite{molev2007yangians}. Chari and Pressley work with $Y(\mathfrak{sl}_n)$ and define maps $ev_z:Y(\mathfrak{sl}_n) \to U(\mathfrak{sl}_n)$ in order to construct modules $V_z(\mu)$, which are equal to $V(\mu)$, where $\mu$ is integral weight of $\mathfrak{sl}_n$ with the structure of a $Y(\mathfrak{sl}_n)$-module given by $ev_z$. Now it is easy to see that the map $ev'_z:Y(\mathfrak{sl}_n) \to U(\mathfrak{gl}_n)$ which is equal to the composition of $ev_z$ with an inclusion of $U(\mathfrak{sl}_n)\to U(\mathfrak{gl}_n)$ is equal to the composition $ev \circ\tau_z \circ j$, where $\tau_z$ is the automorphism of $Y(\mathfrak{gl}_n)$ given by $T(u) \mapsto T(u-z)$ and $j$ is the inclusion of $Y(\mathfrak{sl}_n)$ to $Y(\mathfrak{gl}_n)$. Note that $ev \circ \tau_z$ sends $T(u)$ to $1 + \frac{\sigma}{u-z}$. So if we take $R_f$ to be automorphism from Prop. 2.1.2(e) with $f = 1-z/u$ it follows that $ev \circ R_f \circ \tau_z$ maps $T(u)$ to $1+\frac{R-z}{u}$. So if $\psi_z$ is the automorphism of $U(\mathfrak{gl}_n)$ which sends $\sigma$ to $\sigma-z$ (or, equivelantly, on representations we just tensor with a one-dimensional representation of $\mathfrak{gl}_n$), then we have $ev \circ R_f \circ \tau_z  = \psi_z \circ ev$. But since $Y(\mathfrak{sl}_n)$ is invariant under $R_f$, it follows that $ev'_z =  \psi_z \circ ev \circ j$. But this means that to obtain the modules $V_z(\mu)$ we can equivalently consider all possible $\mathfrak{gl}_n$ modules $V(\lambda)$ with $\lambda$ being a weight of $\mathfrak{gl}_n$ equal to $\mu$ when restricted to $\mathfrak{sl}_n$. So if we consider all $\mathfrak{gl}_n$ modules, to get all evaluation modules, so it is enough to use only the map $ev$, which Molev does in \cite{molev2007yangians} and we will do here.
\end{rem}

Now we are ready to state the theorem (Cor. 3.4.2 in \cite{molev2007yangians}):
\begin{thm}
Irreducible finite dimensional representations of $Y(\mathfrak{gl}_n)$ are in $1-1$ correspondence with  tuples $(f(u), P_1(u), \dots, P_{n-1}(u))$, where $P(u)$ are monic polynomials in $u$ and $f(u) \in 1 + u^{-1}\mathbb C[[u^{-1}]]$. Moreover, $\frac{\lambda_i}{\lambda_{i+1}} = \frac{P_i(u+1)}{P_i(u)}$. Also up to tensoring with 1-dimensional representations each irreducible finite-dimensional representation is a subquotient of the tensor product of evaluation modules.
\end{thm}
\begin{def0}
The polynomials $P_i$ are called the Drinfeld polynomials.
\end{def0}
We have an analogous result for $Y(\mathfrak{sl}_n)$ (following Cor. 3.4.8. in \cite{molev2007yangians}):
\begin{thm}
Finite dimensional representations of $Y(\mathfrak{sl}_n)$ are in $1-1$ correspondence with tuples $(P_1(u), \dots, P_{n-1}(u))$ of Drinfeld polynomials. Moreover each such representation is a subquotient of a tensor product of  evaluation modules.
\end{thm}

\section{Yangians $Y(\mathfrak{gl}_n)$ and $Y(\mathfrak {sl}_n)$ and their representations in positive characteristic.}

Note that one can define the Yangians for special and general linear groups over $\overline{\mathbb F}_p$ exactly in the same way as above. One can use exactly the same arguments to show that all statements in Prop. 2.1.2, except $(d)$ hold. The problem is that the center of the Yangian in positive characteristic is bigger than in zero characteristic, it is calculated in \cite{brundan2017p}. But nevertheless one can define $Z_{HC}(Y(\mathfrak{gl}_n))$ as the subalgebra generated by the coefficients of the series $qdet \ T(u)$ and then for $p>n$ an analog of Prop. 2.1.3(c) still holds. Indeed, according to Thm. 6.1 in \cite{brundan2017p} we have $Y(\mathfrak{gl}_n) = Y(\mathfrak{sl}_n)\otimes Z_{HC}(Y(\mathfrak{gl}_n))$.

In order to study representations of the Yangian in complex rank for algebraic $t$, according to Thm. 1.3.1(b), we need to know something about the representations of Yangians in positive characteristic for a sufficiently large $p$. To do this, we first need some results about $\mathfrak{gl}_n$ representations in positive characteristic. We will discuss this in the following subsection, then in the next subsection we will return to Yangians in positive characteristic.

\subsection{Representations of $\mathfrak{gl}_n$ and $\mathfrak{sl}_n$ in positive characteristic.}

Fix an irreducible  $\mathfrak{sl}_n$-representation $V(\lambda)$ over $\mathbb C$. This representation is obtained as a quotient of the Verma module $M(\lambda)$ by the subrepresentation generated by $f_i^{\lambda_i+1}v_\lambda$, where $v_{\lambda}$ is a highest weight vector of the Verma module.

Now in positive characteristic, we can also consider the Verma module $M(\lambda,p)$, and all the vectors $f_i^{\lambda_{i}+1}v_\lambda$ are going to be singular, because the coefficients for the action of $e_j$ on these vectors depend only on $\lambda$ and the structure  constants of the Lie algebra, which are all integers. Hence if the coefficients are zero in characteristic 0, they are going to stay zero after reduction. So we still have  quotient modules $V(\lambda,p)$ defined in the same way, called the Weyl modules, and the only question is whether they are irreducible.

First consider the set $Q$ of weights appearing in $V(\lambda)$. We want to find a restriction on $p$ such that no two elements of $Q$ differ by a linear combination of simple roots with coefficients divisible by $p$. Suppose that two weights $\mu_1,\mu_2$ from $Q$ differ by  a lsum $\sum_i p c_i \alpha_i$. Since $Q$ is $W$-symmetric, it follows that $\forall \mu \in Q$ we have $\lambda_1 \ge \mu_i \ge \lambda_n$. Thus  $|(\mu_1)_i - (\mu_2)_i| \le \lambda_1 - \lambda_n$. Since it is at the same time divisible by $p$, we must conclude that for $p > \lambda_1-\lambda_n+1$ we have $\mu_1 = \mu_2$.

Hence if we consider $V(\lambda, p)$ for $p > \lambda_1 - \lambda_n + 1$, it makes sense to speak about it as a highest weight module, since no two elements of $Q$ get identified when we pass to positive characteristic and no two elements are connected by a root which were not connected before passing to positive characteristic. Moreover we can use the following definition of the weight order: $\mu < \mu'$ iff $\exists \mu_i$ such that $\mu_0 = \mu$, $\mu' = \mu_n$ and $\mu_i = \mu_{i+1} - \alpha_j$ for some $j$. This definition is equivalent to the standard one in characteristic zero. In positive characteristic under this definition $\lambda$ is still the maximal weight of $V(\lambda,p)$ since all weights $\lambda + \alpha_j$ are not in $Q$ even modulo $p$.

It is known (Thm. 1 in \cite{humphreys1971modular} or Thm. 2 in \cite{KAC1976136}) than in positive characteristic one can define a generalization of the Harish-Chandra center and central characters $\chi_\lambda$ in such a way that the following theorem holds:
\begin{thm}
For $p \ne 2$ one has $\chi_\lambda = \chi_\mu$ iff $w(\lambda + \rho) = \mu + \rho$  modulo $p$, for $w  \in W$.
\end{thm}

So we have the following lemma.

\begin{lemma}
  For $p > \lambda_1 -\lambda_n + n$ the module $V(\lambda, p)$ is irreducible.
\end{lemma}

\begin{proof}
Suppose we have a submodule $N\subset V(\lambda,p)$. Since already in $V(\lambda,p)$ the order $<$ on weights has no cycles for such $p$, it follows that it is a well-defined order when restricted to $N$, hence $N$ has a highest weight vector. Thus $V(\lambda,p)$ has a non-trivial singular vector with weight $\mu$. But by the above $\mu$ should be equal to $w(\lambda + \rho) - \rho$ for some $w \in W$. So we only need to prove that there is no element of $Q$ which differs from $w(\lambda + \rho) - \rho$ for some $w$ by a linear combination of simple roots multiplied by $p$. So suppose $ \lambda \ge \mu \ge w_0(\lambda)$ and $\mu - w(\lambda + \rho) - \rho$ is divisible by $p$. 

Write down $\lambda = (\lambda_1, \dots, \lambda_n)$ and $w_0(\lambda) = (\lambda_n, \dots, \lambda_1)$, where  $\lambda_1 \ge \lambda_2 \ge \dots \ge \lambda_n$.  Now $w(\lambda + \rho) - \rho = (\lambda_{w(1)} + e_1, \dots, \lambda_{w(n)} + e_n)$, where $e_j$ are some integers $n > e_j > -n$. Now, since $\mu \le \lambda$, it follows that $\mu_1 \le \lambda_1$ and $\mu_n \ge \lambda_n$. Since $Q$ is $W$-stable, it follows that $\lambda_1 \ge \mu_i \ge \lambda_n$. Now $\mu_i = \lambda_{w(i)} + e_i \ mod \ p$, in other words a number between $\lambda_1$ and $\lambda_n$ differs from a number between $\lambda_1 + n-1$ and $\lambda_n -n+1$ by a multiple of $p$, but  $\lambda_1 - \lambda_n +n < p$ and it follows that $\mu_i = \lambda_{w(i)} + e_i$. But this contradicts $\mu \in Q$. 

So we proved that there are no non-trivial singular vectors in $V(\lambda,p)$ and hence, no nontrivial submodules.
\end{proof}

The same statement of course holds for $\mathfrak{gl}_n$-modules.

\subsection{Representations of $Y(\mathfrak{gl}_n)$ and $Y(\mathfrak{sl}_n)$ in positive characteristic.}

Using the above tools we can try and repeat some of the arguments from the classification of irreducible representations of the Yangian  following \cite{molev2007yangians}. First we need to understand in which sense we can treat a Yangian representation in positive characteristic as a highest weight representation.

\begin{prop}
Consider an irreducible  representation L of $Y(\mathfrak{gl}_n)$,  which as a representation of $\mathfrak{gl_n}$ is equal to the sum $V(\lambda_1,p) \oplus \dots \oplus V(\lambda_k,p)$. Then for \\ $p > \max((\lambda_j)_1-(\lambda_j)_2)+n$, $L$ has a unique up to scaling singular vector, whose $\mathfrak{gl_n}$ weight is maximal.
\end{prop}

\begin{proof}
The problem here is to rule out the possibility of $L$ not having any singular vectors at all.
But since we can speak about the underlying $\mathfrak{gl}_n$ representation as highest weight, we can repeat the argument of Thm. 3.2.7 in \cite{molev2007yangians} without any changes.
\end{proof}

So it follows that for big enough $p$, $L = L(\lambda(u),p)$ for some $\lambda(u)$ (where $L(\lambda(u),p)$ is the irreducible quotient of the Verma module $M(\lambda(u),p)$).

Now let's move to the important case $Y(\mathfrak{gl}_2)$.

\begin{prop}
Suppose $L$ is a finite-dimensional representation of $Y(\mathfrak{gl}_2)$ isomorphic to $V(\lambda_1,p) \oplus \dots \oplus V(\lambda_k,p)$ as a $\mathfrak{gl}_2$ representation, then for $p > \max((\lambda_j)_1 - (\lambda_j)_2)+1$ we have $L = L(\lambda(u),p)$ and there is a formal series $f(u) \in 1 + u^{-1}\overline{\mathbb F}_p[[u^{-1}]]$ such that $f(u)\lambda_1(u)$ and $f(u)\lambda_2(u)$ are polynomials in $u^{-1}$.
\end{prop}
\begin{proof}
By the previous discussion under our assumption  we can repeat the proof of Proposition 3.3.1 of \cite{molev2007yangians} for $L$ without any change.
\end{proof}

For $n \in \mathbb F_p$ denote by $[n] \in \mathbb Z$  the minimal element of $\mathbb Z_{\ge 0}$ such that $[n] \ mod \ p = n$.

Now for $\alpha,\beta \in \overline{\mathbb F}_p$ denote by $L(\alpha,\beta,p)$ the irreducible  quotient of the Verma module $M(\alpha,\beta,p)$ for $\mathfrak{gl}_2$. By a direct calculation $L(\alpha,\beta,p)$ has a basis $f^kv$ for $k$ from $0$ to $l$, where $l$ is equal to $p-1$ if $\alpha-\beta \notin \mathbb F_p$ and equal to $[\alpha-\beta]$ if $\alpha-\beta \in \mathbb F_p$.

\begin{prop}
 Given two sequences $\alpha_i,\beta_i$ of elements of $\overline{\mathbb F}_p$ for $i =1,\dots,k$, re-numerate them in such a way that $[\alpha_i-\beta_i]$ is minimal among all $[\alpha_j-\beta_k]$ for $i \le j,k$ if defined, and if not defined then all $\alpha_j-\alpha_k \notin \mathbb F_p$. Then the representation
 $$
 L(\alpha_1,\beta_1,p) \otimes \dots \otimes L(\alpha_k,\beta_k,p)
 $$
  is irreducible.
\end{prop}

\begin{proof}
We can use the proof of Prop. 3.3.2 in \cite{molev2007yangians} with slight changes (in particular, the number $p$ used in the proof we will denote by $q$).

First let's explain how the re-numeration of $\alpha_i$, $\beta_i$ works. Consider all possible pairs $i,j$ such that $[\alpha_i-\beta_j]$ is defined, and choose a pair such that $[\alpha_i-\beta_j]$ is minimal among them. These two elements now will become new $\alpha_1$ and $\beta_1$. Then repeat for all remaining $\alpha_i$ and $\beta_j$.

Denote the module in question by $L$. Suppose we have a submodule $N \subset L$. By definition $t_{12}(u)$ acts on $L(\alpha_i,\beta_i,p)$ as $eu^{-1}$, so it acts locally nilpotently. Since $\Delta(t_{12}(u)) = t_{11}(u)\otimes t_{12}(u) + t_{12}(u) \otimes t_{22}(u)$, it follows that $t_{12}(u)$ acts locally nilpotently on $L$ and $N$. Hence $N$ has a vector singular with respect to $t_{12}(u)$. Thus if we prove that $L$ has only one singular vector (the tensor product of singular vectors of all $L(\alpha_i,\beta_i,p)$) it will follow that this module has no  nontrivial submodules.

We will prove this claim by induction. So suppose we have $\xi = \sum_{r=0}^q e^r\zeta_1 \otimes \xi_r$. Here $\zeta_1$ is a singular vector of $L(\alpha_1,\beta_1,p)$ and $q$ is an integer less or equal to $p-1$ or $[\alpha_i-\beta_i]$ if it is defined. Then repeating all the steps of the proof of Prop. 3.3.2 in \cite{molev2007yangians} we obtain a formula:
$$
q(\alpha_1-\beta_1-q+1)(\alpha_1-\beta_2-q+1)\dots(\alpha-\beta_k-q+1) = 0 \ . 
$$
If $\alpha_1-\beta_1 \notin \mathbb F_p$, then all $\alpha_1-\beta_k \notin \mathbb F_p$, hence the equation is satisfied only for $q=0$. If $[\alpha_1-\beta_1] = k \le p-1$, then $q \le k$. Hence $[\alpha_1-\beta_1]-[q]+1 = k-[q]+1$ lies between $0$ and $k+1$. So it may be equal to zero only if $k=p-1$ and $q=0$. All other $\alpha_1-\beta_j$ are either not in $\mathbb F_p$ and hence the corresponding brackets are not zero, or $[\alpha_1-\beta_j]\ge k$ and hence $[\alpha_1-\beta_j]+1-q$ can be zero again only for $q=0$. Hence it follows that $q = 0$ and the singular vector is equal to $\zeta_1 \otimes \dots \otimes \zeta_k$ up to scaling.

Now the fact that this singular vector generates $L$ is proved in the same way as in characteristic zero. Hence $L$ is irreducible.
\end{proof}

From Prop. 3.2.2 we know that any finite-dimensional irreducible $Y(\mathfrak{gl}_2)$-module $L$ satisfying the condition stated there after tensoring with a one-dimensional representation is isomorphic to $L' = L(\lambda(u),p)$, where $\lambda(u)$ is a pair of polynomials. Write down $\lambda_1(u) = (1+\alpha_1u^{-1})\dots(1+\alpha_ku^{-1})$ and $\lambda_2(u) = (1+\beta_1u^{-1})\dots(1+\beta_ku^{-1})$. Here $\alpha_i,\beta_i$ are ordered in a way consistent with Prop. 3.2.3. Now consider $L(\alpha_1,\beta_1,p)\otimes \dots \otimes L(\alpha_k,\beta_k,p)$. By the above discussion it follows that this module is also isomorphic to $L(\lambda(u),p)$. But this gives us another condition on $\alpha_i,\beta_i$. Since we know that in $L'$ there is no chains of successive weights differing by the action of $e$ of length $p-1$, it follows that no $L(\alpha_i,\beta_i,p)$ can have dimension $p-1$. So all $\alpha_i-\beta_i \in \mathbb F_p$. 

Now we are able to prove the following theorem (for the corresponding theorem in characteristic zero see Thm.3.3.3 in \cite{molev2007yangians}):
\begin{thm}
Suppose $L$ is a representation satisfying the assumption of Prop. 3.2.2  which is isomorphic to $L(\lambda(u),p)$. Then there is a monic polynomial $P(u)$ in $u$ such that:
$$
\frac{\lambda_1(u)}{\lambda_2(u)} = \frac{P(u+1)}{P(u)} \ .
$$
\end{thm}
\begin{proof}
The proof is easy. By the assumption $L = L(\nu(u),p)$. By Prop. 3.2.2 we can make $\lambda(u) =\nu(u)f(u)$ be a pair of polynomials with roots $\alpha_i$ and $\beta_i$ respectively. By the above we know that $\alpha_i-\beta_i \in \mathbb F_p$, hence we can take:
$$
P(u) = \prod_{i=1}^k (u+\beta_i)(u+\beta_i+1)\dots(u+\alpha_i-1) \ .
$$
\end{proof}

Let's generalize this to $Y(\mathfrak{gl}_n)$-modules in positive characteristic.

\begin{thm}
Suppose $L$ is a representation satisfying the assumptions of  Prop. 3.2.1, then there are monic polynomials $P_i(u)$ such that:
$$
\frac{\lambda_i(u)}{\lambda_{i+1}(u)} = \frac{P_i(u+1)}{P_i(u)} \ ,
$$

\end{thm}
\begin{proof}
We know that $L = L(\lambda(u))$. Now using the inclusion $Y(\mathfrak{gl}_2)\rightarrow Y(\mathfrak{gl}_n)$, where $t_{i,j} \rightarrow t_{i+k,j+k}$, we may regard $L$ as a $\mathfrak{gl}_2$-module. The $\mathfrak{sl}_2$-weights with respect to this inclusion lie between $\lambda_1-\lambda_n$ and $\lambda_n-\lambda_1$, so the assumption of Thm. 3.2.4 holds, hence $\frac{\lambda_k}{\lambda_{k+1}} = \frac{P_k(u+1)}{P_k(u)}$, for some $P_k$. 
\end{proof}

We also want to be able to construct a finite-dimensional representation with given Drinfeld polynomials. 
\begin{thm}
Set $p>2$.
Suppose we have a collection of monic polynomials \\ $P_1(u), \dots, P_{n-1}(u)$. Then there is a finite-dimensional representation $L(\mu(u),p)$ such that:
$$
\frac{\mu_i(u)}{\mu_{i+1}(u)} = \frac{P_i(u+1)}{P_i(u)} \ .
$$
\end{thm}
\begin{proof}
The second part of the proof of Thm. 3.4.1 in \cite{molev2007yangians} can be repeated without any problem.

Indeed each of $L(\mu^{(k)},p)$ has $\mu^{(k)}_1 - \mu^{(k)}_n = 1$, so they are equal to $V(\mu^{(k)},p)$. Now $L(\mu^{(1)},p) \otimes \dots \otimes L(\mu^{(k)},p)$ has an external grading induced from the characteristic zero case which is consistent with the $Y(\mathfrak{gl}_n)$-action. So it follows that $L(\mu(u),p)$ is a subquotient in this finite-dimensional module.
\end{proof}

\begin{rem}
 Note that the weight with maximal $\lambda_1-\lambda_2$ appearing in $L(\mu^{(1)},p) \otimes \dots \otimes L(\mu^{(k)},p)$ is equal to $\nu \sum_k \mu^{(k)}$. Moreover, $\nu_1-\nu_n = \sum_i \deg \ P_i$. Thus for $\sum \deg \ P_i + n < p$ it follows that the corresponding $L(\mu(u))$ satisfies the conditions of Prop. 3.2.1.
\end{rem}
\begin{rem}
Another difference between the positive characteristic case and zero characteristic case is the fact that $P_i$ in Thm. 3.2.5 are not unique. Indeed from $\frac{Q_i(u+1)}{Q_i(u)} = \frac{P_i(u+1)}{P_i(u)}$, it follows that $\frac{P_i}{Q_i} = F_i$ satisfies $F_i(u)=F_i(u+1)$. Thus $F_i$ is a ratio of products of expressions of the form $(u+c)(u+1+c) \dots (u+c+p-1) = (u+c)^p-(u+c)$ for some $c \in \overline{\mathbb F}_p$. Further, we set $q_p(u) := u^p-u$.
\end{rem}

This also shows the following:
\begin{cor}
Suppose $L$ is a representation satisfying the assumptions of Prop. 3.2.1. Then it is a subquotient of a tensor product of evaluation representations.
\end{cor}
\begin{proof}
Apply Thm. 3.2.5 and then the proof of Thm. 3.2.6.
\end{proof}

\begin{cor}
Finite-dimensional representations of $Y(\mathfrak{gl}_n)$ satisfying the condition of Prop. 3.2.1 are classified by tuples $(f(u),[P_1],\dots,[P_{n-1}])$, where $f(u) \in 1 + u^{-1}\overline{\mathbb F}_p[[u^{-1}]]$, $\sum \deg \ P_i + n < p$ and $[ \ ]$ denotes the equivalence classes generated by relation  $P_i \equiv Q_i$ if $Q_i = P_iq_p(u+c)$.
\end{cor}
\begin{cor}
Finite-dimensional representations of $Y(\mathfrak{sl}_n)$ satisfying the condition of Prop. 3.2.1 are classified by tuples $[P_1],\dots,[P_{n-1}]$, where $\sum_i \deg (P_i) + n < p$ and $[ \ ]$ denotes the equivalence classes generated by relation  $P_i \equiv Q_i$ if $Q_i = P_iq_p(u+c)$.
\end{cor}

\section{Yangians in complex rank}

\subsection{Yangians $Y(\mathfrak{gl}_t)$ and $Y(\mathfrak{sl}_t)$}

For this subsection fix $t \in \mathbb C \backslash \mathbb Z$.

Below $t_n$ are the same as in Thm. 1.3.1(b) for algebraic $t$ and $t_n = n$ for transcendental $t$. Similarly $p_n$ are as in Thm. 1.3.1(b) for algebraic $t$ and $p_n = 0$ for transcendental $t$. Also $\overline{\mathbb F}_0 := \overline{\mathbb Q}$. 

To define $Y(\mathfrak{gl}_t)$ we will mimic the Faddeev-Reshetikhin-Takhtajan presentation of $Y(\mathfrak{gl}_n)$ in the following way (following section 7 in \cite{etingof2016representation}). 

For $i \in \mathbb Z_{> 0}$ denote by $V_i\simeq V$ a collection of objects of $\Rep (GL_t)$. Consider the tensor algebra $A = T(\bigoplus_{i=1}^{\infty}V_i\otimes V_i^*)$. This is an ind-object of $\Rep (GL_t)$. Obviously it is generated as an algebra by the images of the inclusion maps $T_i: V \otimes V^* \rightarrow A$ sending $V\otimes V^*$ to $V_i\otimes V_i^*$. Using this we can define a formal power series in $u^{-1}$ with coefficients in $\Hom(V\otimes V^*, A)$, $T(u) = 1 + \sum_{i>0} T_iu^{-i}$. Also by composing with certain evaluation and coevaluation maps we can regard $T(u)$ as an element of $\Hom(V,V\otimes A)[[u^{-1}]]$. Now define $R(u)\in \Hom(V\otimes V, V \otimes V)[[u^{-1}]]$ in the same way as for integer rank we set: $R(u)=1 +\frac{\sigma}{u}$, where $\sigma$ interchanges factors in the tensor product. Consider an element of $\Hom(V_I \otimes V_{II}, V_I\otimes V_{II} \otimes A)[[u^{-1},v^{-1}]]$ (here $V_I \simeq V_{II} \simeq V$) given by
\begin{equation}
    Q(u,v) = (u-v)(R(u-v)T^I(u)T^{II}(v) - T^{II}(v)T^I(u)R(u-v)) \ ,
\end{equation}
where $i$ in $T^i$ specifies  on which $V_i$ it acts, and $R$ in the first summand acts on $V_I\otimes V_{II} \otimes A$ by $R \otimes \Id$. Again using the evaluation and coevaluation maps, we can think of $Q$ as an element of  $\Hom(V\otimes V^* \otimes V\otimes V^*, A)[[u^{-1},v^{-1}]]$. Now write $Q(u,v) = \sum_{i,j}Q_{i,j} u^{-i}v^{-j}$. We are ready to define the Yangian:
\begin{def0}
The Yangian of $\mathfrak{gl}_t$, $Y(\mathfrak{gl}_t)$, is the algebra obtained as the cokernel of the following map of ind-objects:
$$
\bigoplus_{i,j} Q_{i,j}: \bigoplus_{i,j} V\otimes V^* \otimes V\otimes V^* \rightarrow A \ ,
$$
or, equivalently, is the quotient of $A$ by the quadratic relations given by $Q_{i,j}$.
\end{def0}

\begin{rem}
To specify an algebra homomorphism from $A$, it is sufficient to give a bunch of morphisms from $V\otimes V^*$, since $T_i$ freely generate $A$. To specify an algebra homomorphism from $Y(\mathfrak{gl}_t)$, it is sufficient to give a bunch of maps from $V \otimes V^*$ such that the quadratic relations $Q_{i,j}$ are satisfied, since $T_i$ generate the Yangian.
\end{rem}

The next step is to generalize some important properties  by direct methods, but we will rather use the methods of ultraproducts and Łoś's theorem since these are the main tools of this paper. The main fact which we are going to use is that $\prod_{\mathcal F} Y(\mathfrak{gl}_{t_n})$ is equal to $Y(\mathfrak{gl}_t)$ which follows from the definition, since it uses exactly the same spaces and maps as the finite rank definition (meaning that they are ultraproducts of the spaces/maps in the finite case).

These properties were stated in \cite{etingof2016representation} section 7. Here we provide their proof:

\begin{prop}

\textbf{a)} There is a Hopf algebra structure on $Y(\mathfrak{gl}_t)$ given by $\Delta(T(u)) = T^I(u)T^{II}(u)$ and $S(T(u)) = T(u)^{-1}$. 

\textbf{b)} There is an algebra homomorphism $i:U(\mathfrak{gl}_t) \rightarrow Y(\mathfrak{gl}_t)$, which on $V\otimes V^* \subset U(\mathfrak{gl}_t)$ acts as $T_1$.

\textbf{c)} There is a homomorphism $ev:Y(\mathfrak{gl}_t)\rightarrow U(\mathfrak{gl}_t)$ given by $T(u) \mapsto R(u)$.
\end{prop}

\begin{proof}
The general strategy is the following. Prove that we have some maps by giving an element-free construction for them (this guarantees that our maps are ultraproducts of finite rank maps), which generalizes the finite rank case, and then apply Łoś's theorem to prove that these maps satisfy the required properties. We will spell out the proof of $a)$ in more detail to show how this strategy works.

\textbf{a)} First we define the map $\Delta':A \rightarrow A\otimes A$, using the collection of maps $ \Delta': V\otimes V^* \rightarrow A \otimes A$, which is equal to the sum of maps $T_i^{I}: V \otimes V^* \to Im T^I_i \otimes 1$,  $(T_{i-j}^I\otimes T^{II}_j)\circ coev: V \otimes V^* \to Im T^I_{i-j} \otimes T^{II}_j$ for $1 \le j \le i-1$ and $T^{II}_i: V\otimes V^* \rightarrow 1 \otimes Im T^{II}_i$ (This is just the formula $\Delta(T(u)) = T^{I}(u)T^{II}(u)$ written explicitly).

Now since $Y(\mathfrak{gl}_t)$ is the quotient of $A$, we have a map $A \rightarrow Y(\mathfrak{gl}_t) \otimes Y(\mathfrak{gl}_t)$. But we know that for integer rank Yangians this map factors through the relations $Q_{i,j}$, hence by Łoś's theorem it factors through them for rank $t$ (because "factors through" means that $\Delta'$ satisfies a collection of equations). Hence we have a map $\Delta$.

The map $S$ is constructed in the same way.  We explicitly define a map $A \rightarrow Y(\mathfrak{gl}_t)$ as in the finite rank case using $T(u)^{-1}$ and then argue that it gives us a map for $Y(\mathfrak{gl}_t)$ for the same reason.

Finally the fact that $\Delta$ and $S$ define a Hopf algebra structure on $Y(\mathfrak{gl}_t)$ is equivalent to saying that these maps satisfy some equations. But they satisfy these equations for integer rank, hence by Łoś's theorem they satisfy them for rank $t$.

\textbf{b)} We can consider a homomorphism $T(V\otimes V^*) \rightarrow Y(\mathfrak{gl}_t)$ given by  $T_1:V \otimes V^* \to Im T_1$. By the same argument as before this map factors through $U(\mathfrak{gl}_t)$.

\textbf{c)} The formula $T(u) \mapsto R(u)$ gives as an algebra homomorphism from $A$ to $U(\mathfrak{gl}_t)$. By the same logic it extends to  a map: $Y(\mathfrak{gl}_t) \rightarrow U(\mathfrak{gl}_t)$.
\end{proof}

Now, we can also define the Yangian of the special linear Lie algebra in complex rank. 
We know that in the finite rank case $Y(\mathfrak{sl}_{t_n})$ is defined to be the subalgebra of the Yangian for $\mathfrak{gl}_{t_n}$ invariant under automorphisms $T(u) \mapsto f(u)T(u)$ for $f \in 1 + u^{-1} \mathbb K[[u^{-1}]]$ (Def 2.1.2). 
To mimic this definition in the rank $t$ case, we need to understand how the ultraproduct of a collection of such automorphisms looks like:
\begin{lemma}
Consider a collection of $f^{(n)} \in 1 + u^{-1} \overline{\mathbb F}_{p_n}[[u^{-1}]]$. The ultraproduct of the automorphisms of the Yangians given by such a collection is the automorphism of $Y(\mathfrak{gl}_t)$ given by some $f \in 1 + u^{-1}\mathbb C[[u^{-1}]]$, denoted by $R_f$.
\end{lemma}

\begin{proof}
Let $f^{(n)} = 1 + \sum f^{(n)}_i u^{-i}$. The automorphisms are given explicitly as $T_i \mapsto T_i + f^{(n)}_1T_{i-1} + \dots + f^{(n)}_i$. Now the important fact, which we are going to use here, is the isomorphism $\prod_{\mathcal F} \overline{\mathbb F}_{p_n} \simeq \mathbb C$. Since this isomorphism is fixed, it follows that the ultraproduct of a collection of maps as above with $f^{(n)}_i \in \overline{\mathbb F}_{p_n}$ will give us a similar map with $f_i \in \mathbb C$. Thus we obtain $f \in 1 + u^{-1}\mathbb C[[u^{-1}]]$  and automorphism $R_f$ sending $T(u)$ to $f(u)T(u)$.
\end{proof}

\begin{def0}
The Yangian $Y(\mathfrak{sl}_t)$ is defined as the subalgebra of $Y(\mathfrak{gl}_t)$ of invariant "elements" under all automorphisms $R_f$.
\end{def0}

\begin{rem}
 Here by "elements" of $Y(\mathfrak{gl}_t)$ we mean the following. Suppose $X$ is an object of a tensor category $\mathcal C$. Also suppose that we have an action of group $G$ on $X$, i.e. a homomorphism $G \to \Aut(X)$. Then we can define $X^G$ as a interesction $\cap_{g\in G} \ker(1-g)$. So if $G$ is a group of automorphisms $R_f$, then we define $Y(\mathfrak{sl}_t) = Y(\mathfrak{gl}_t)^G$.
\end{rem}

Now it is easy to see that $Y(\mathfrak{sl}_t)$ defined in such a way is the ultraproduct of $Y(\mathfrak{sl}_{t_n})$. Indeed, by  Lemma 4.1.3 $Y(\mathfrak{sl}_t)$ is the subalgebra in the ultraproduct of Yangians of invariants under all possible collections of automorphisms of $Y(\mathfrak{gl}_{t_n})$ given by $f^{(n)}$, hence the ultraproduct of $Y(\mathfrak{sl}_{t_n})$. Also we have the following corollary: 

\begin{cor}
\textbf{a)} The subalgebra $Y(\mathfrak{sl}_t)$ is a Hopf subalgebra of $Y(\mathfrak{gl}_t)$.

\textbf{b)} The map $i$ restricts to the map $i:U(\mathfrak{sl}_t) \rightarrow Y(\mathfrak{sl}_t)$.

\textbf{c)} We have $Y(\mathfrak{gl}_t) = Y(\mathfrak{sl}_t) \otimes \mathbb C[z_1,z_2,\dots ]$. There the second factor is the sum of trivial representations as an object of $\Rep(GL_t)$.
\end{cor}

\begin{proof}
\textbf{a)}We know that this proposition holds for integer rank. Hence we know that the image of invariant elements under the maps $\Delta$ and $S$ is invariant for integer rank, hence it is invariant for complex rank by Łoś's theorem.

\textbf{b)} This is obvious since $R_f$ restricted to $U(\mathfrak{gl}_t)$ is just the automorphism which sends $V\otimes V^*$ into itself via $\Id + f_1$\footnote{Here by $f_1$ we mean $f_1$ times the projector $coev \circ ev$.}.  Hence the invariants in $U(\mathfrak{gl}_t)$ are exactly $U(\mathfrak{sl}_t)$.

\textbf{c)} From Prop. 2.1.3(c) and discussion in the beginning of section 3 it follows that for almost all $n$ we have $Y(\mathfrak{gl}_{t_n}) = Y(\mathfrak{sl}_{t_n})\otimes Z_{HC}(Y(\mathfrak{gl}_t))$, where $Z_{HC}(Y(\mathfrak{gl}_t)) = \overline{\mathbb F}_{p_n}[z_1,z_2,\dots, ]$. This decomposition respects the action of $\mathfrak{gl}_{t_n}$ on $Y(\mathfrak{gl}_{t_n})$. So it follows that $Y(\mathfrak{gl}_t) = \prod_{\mathcal F}Y(\mathfrak{gl}_{t_n}) = [\prod_{\mathcal F}Y(\mathfrak{sl}_{t_n})]\otimes [\prod_{\mathcal F} \overline{\mathbb F}_{p_n}[z_1,z_2,\dots ]] = Y(\mathfrak{sl}_t)\otimes \mathbb C[z_1,z_2,\dots]$. 
\end{proof}

\begin{rem}
 Note that Cor. 4.1.5(c) gives us an alternative way to define $Y(\mathfrak{sl}_{t})$, clarifying the concept of "invariant elements" mentioned before.
\end{rem}

\subsection{Finite-length representations of $Y(\mathfrak{gl}_t)$, $Y(\mathfrak{sl}_t)$.}

First we define the category of $Y(\mathfrak{gl}_t)$-modules which we are interested in.

\begin{def0}
Denote by $\Rep_0(Y(\mathfrak{gl}_t))$ the category with objects being objects $M \in \Rep (GL_t)$ together with an element $\mu_M \in \Hom(Y(\mathfrak{gl}_t)\otimes M,M)$ such that:

\textbf{a)} $M$ has finite length.

\textbf{b)} $M$ is a representation of $Y(\mathfrak{gl}_t)$, i.e. $\mu_M \circ (1 \otimes \mu_M) = \mu_M \circ (\mu \otimes 1)$ as elements of $\Hom(Y(\mathfrak{gl}_t)\otimes Y(\mathfrak{gl}_t)\otimes M,M)$, where $\mu$ is the product map of the Yangian.

\textbf{c)} The map $\mu_M \circ (i \otimes 1)$ gives the standard structure of a $\mathfrak{gl}_t$ representation on $M$.

The morphisms in $\Rep_0(Y(\mathfrak{gl}_t))$ are morphisms of $\Rep(GL_t)$ which commute with the representation structure.

Also we will denote by $\Rep_0'(Y(\mathfrak{gl}_t))$ we category defined in the similar way, but we require only the $\mathfrak{sl}_t$-action to be standard in $(c)$.

\end{def0}

Similarly one can define the category $\Rep_0(Y(\mathfrak{sl}_t))$(using Cor. 4.1.5). We have the following result connecting these categories to the categories of finite-dimensional representations of finite rank Yangians over $\overline{\mathbb F}_{p_n}$ denoted by $\textbf{Rep}_0(Y(\mathfrak{gl}_{t_n}),p_n)$ and $\textbf{Rep}_0(Y(\mathfrak{sl}_{t_n}),p_n)$:

\begin{lemma}
 The category $\Rep_0(Y(\mathfrak{gl}_t))$ is the full subcategory of $\prod_{\mathcal F} \textbf{Rep}_0(Y(\mathfrak{gl}_{t_n}),p_n)$ given by the objects whose image in $\prod_{\mathcal F} \textbf{\Rep}_{p_n}(GL_{t_n})$ is contained in $\Rep(GL_t)$. Moreover, the irreducible representations of $Y(\mathfrak{gl}_t)$ correspond to collections of representations of $Y(\mathfrak{gl}_{t_n})$, such that almost all of them are irreducible. So $\Irr(\Rep_0(Y(\mathfrak{gl}_{t_n}),p_n)) = \prod_{\mathcal F} \Irr(\textbf{Rep}_0(Y(\mathfrak{gl}_{t_n}),p_n))$.
\end{lemma}
\begin{proof}
Take $M \in \Rep_0(Y(\mathfrak{gl}_t))$.
As we know, $\Rep(GL_t)$ is a full subcategory of the category $\prod_{\mathcal F} \textbf{Rep}_{p_n}(GL_{t_n})$. Hence because $M$ is an object of $\Rep(GL_t)$, it follows that we have a corresponding collection of objects $M_n \in \textbf{Rep}_{p_n}(GL_{t_n})$. Now, since we have a map $\mu_M \in \Hom(Y(\mathfrak{gl}_t)\otimes M,M)$ it follows that we have a collection of maps $\mu_{M_n} \in \Hom(Y(\mathfrak{gl}_{t_n})\otimes M_n,M_n)$, and by \lthm it follows that almost all $\mu_{M_n}$ satisfy the equation $\mu_{M_n} \circ (1 \otimes \mu_{M_n}) = \mu_{M_n} \circ (\mu \otimes 1)$ giving a structure of a $Y(\mathfrak{gl}_{t_n})$-module to $M_n$. Also, since $\mu_M \circ (i \otimes 1)$ gives a standard structure of a $\mathfrak{gl}_t$-module to $M$, by \lthm, almost all $M_n$ have a standard structure of a $\mathfrak{gl}_{t_n}$-module from $\mu_{M_n} \circ (i \otimes 1)$. So objects of $\Rep_0(Y(\mathfrak{gl}_t))$ indeed correspond to some objects of $\prod_{\mathcal F} \textbf{Rep}_0(Y(\mathfrak{gl}_{t_n}))$. Now to show that this is indeed a full subcategory, we need to look at morphisms. But since $\Rep(GL_t)$ is a full subcategory in $\prod_{\mathcal F}\textbf{Rep}_{p_n}(GL_{t_n})$, it follows that each morphism in $\Rep_0(Y(\mathfrak{gl}_t))$ gives us a unique sequence of morphisms in $\textbf{Rep}_{p_n}(GL_{t_n})$ which commute with the $Y(\mathfrak{gl}_{t_n})$-action for almost all $n$ by Łoś's theorem. Hence this is indeed a full subcategory.

Now, in the other direction, if we have an object of the intersection of $\prod_{\mathcal F} \textbf{Rep}_0(Y(\mathfrak{gl}_{t_n}))$ with $\Rep (GL_t)$, i.e. a sequence of $M_n$ with a structure of $Y(\mathfrak{gl}_{t_n})$-modules s.t. $M=\prod_{\mathcal F}M_n$ lies in $\Rep(GL_t)$, it follows by \lthm that $M$ has a required structure of a $Y(\mathfrak{gl}_t)$-representation.

Let us prove the second statement. Suppose that $M$ is reducible, then we have a non-zero injective morphism from $N \ne M$ to $M$. Thus we have a sequence of $N_n$ and a sequence of maps $N_n \rightarrow M_n$ such that for almost all $n$ these maps are  $Y(\mathfrak{gl}_{t_n})$-module maps, $N_n \ne M_n$ and the maps are injective and non-zero. Hence almost all $M_n$ are reducible. And vice versa, if almost all $M_n$ are reducible, we have a collection of $Y(\mathfrak{gl}_{t_n})$-modules $N_n$ and a collection of injective maps $N_n \rightarrow M_n$. Since $N_n$ is a subobject of $M_n$ as objects of $\textbf{Rep}_{p_n}(GL_{t_n})$, it follows that $N = \prod_{\mathcal F}N_n$ lies in $\Rep(GL_t)$. Also by the above it lies in $\Rep_0(Y(\mathfrak{gl}_t))$. Finally, by \lthm the resulting map $N \rightarrow M$ is injective non-zero and $N\ne M$, hence $M$ is reducible.
\end{proof}

The same lemma can be stated for $Y(\mathfrak{sl}_t)$.
\begin{lemma}
 The category $\Rep_0(Y(\mathfrak{sl}_t))$ is a full subcategory of $\prod_{\mathcal F} \textbf{Rep}_0(Y(\mathfrak{sl}_{t_n}))$ given by the objects whose image in $\prod_{\mathcal F} \textbf{\Rep}_{p_n}(GL_{t_n})$ is contained in $\Rep(GL_t)$. Moreover the irreducible representations of $Y(\mathfrak{sl}_t)$ correspond to collections of representations of $Y(\mathfrak{sl}_{t_n})$, such that almost all of them are irreducible. 
\end{lemma}
\begin{proof}
The proof is exactly the same.
\end{proof}

Recall the classification theorem of finite-dimensional irreducible  $Y(\mathfrak{sl}_n)$-modules. The important step in the proof of this classification is to show that all such modules are subquotients in the tensor product of evaluation modules (we will use $\sqsubset$ to denote subquotients). We want to prove the same thing in our case. To do this, first define evaluation modules. By $\mathbb C(c)$, for a complex number $c$, we will denote a representation of $\mathfrak{gl}_t$ where it acts through its projection to $\mathbb C$ and then by multiplication by $c$. 
\begin{def0}
For a fixed bipartition $\lambda$ and complex number $c$ denote by $L(\lambda+c)$ an object $\Rep'_0(Y(\mathfrak{gl}_t))$ which is equal to $V(\lambda)\otimes \mathbb C(c)$ as an object of $\Rep(GL_t)$ and the $Y(\mathfrak{gl}_t)$-module structure is obtained from $\mathfrak{gl}_t$-module structure by $ev: Y(\mathfrak{gl}_t) \rightarrow U(\mathfrak{gl}_t)$(see Prop. 4.1.2 (c)). We will call such modules evaluation modules.
\end{def0}

Since $Y(\mathfrak{gl}_t)$ is a Hopf algebra, we can freely take tensor products of such modules. Also the structure of a $Y(\mathfrak{gl}_t)$ representation gives them a structure of a $Y(\mathfrak{sl}_t)$ representation as well. 

We can now generalize the subquotient statement for complex rank.

We will need the following lemma:
\begin{lemma}
  Fix $c \in \mathbb Z$, then in the case of algebraic $t$ for almost all $n$ we have $p_n-t_n > c$.
\end{lemma}
\begin{proof}
Suppose $q\in \mathbb Z[x]$ is a minimal polynomial for $t$ with positive leading coefficient. Now since $t_n$ is a root if $q$ modulo $p_n$, it follows that $t_n-p_n$ is also a root modulo $p_n$. So we have $q(t_n-p_n) = A_n p_n$ for some $A_n \in \mathbb Z$. Note that if $A_n = 0$, it follows that $q$ has an integer root, therefore is not minimal. Now we have $|q(x)| < Cx^k$ for $k = \deg \ q$ and $|x|>1$, where $C$ is the sum of absolute values of all coefficients plus 1. Pick $N$ such that $Cc^k < p_N$. Then for all $n \ge N$ if $p_n-t_n \le c$, it follows that $|q(t_n-p_n)| < Cc^k < p_n$, hence $A_n = 0$, which is a contradiction. Thus for all $n \ge N$ it follows that $p_n-t_n > c$.
\end{proof}

\begin{prop}
If $M \in \Rep_0(Y(\mathfrak{sl}_t))$ is irreducible, then there exist non-empty bipartitions $\eta_1, \dots, \eta_k$ and complex numbers $c_1,\dots,c_k$ such that M is a subquotient of $L(\eta_1+c_1)\otimes \dots \otimes L(\eta_k+c_k)$.
\end{prop}

\begin{proof}
Consider the collection of $Y(\mathfrak{sl}_{t_n})$-modules $M_n$ corresponding to $M$. As we know, for almost all $n$ these modules are irreducible.  $M$ as an object of $\Rep(GL_t)$ is equal to a finite sum of $V(\mu_j)$ -- the simple objects corresponding to bipartitions $\mu_j$. Thus for almost all $n$ we have $M_n = \oplus_{j=1}^{K} V((\mu_j)|_{t_n},p)$. Since for big enough $n$ we have $((\mu_j)|_{t_n})_1 - ((\mu_j)|_{t_n})_{t_n} = (\mu_j^\bullet)_1 - (\mu_f^\circ)_1$, it follows that this parameter does not depend on $n$. Hence by Lemma 4.2.5 it follows that for almost all $n$, $M_n$ satisfies the condition of Prop. 3.2.1, and hence by Cor. 3.2.7 we deduce that $M_n$ is a subquotient of evaluation modules: $M_n \sqsubset L(\lambda^{(n)}_1) \otimes \dots \otimes L(\lambda^{(n)}_{k_n})$, where $\lambda^{(n)}_i$ is a dominant weight. Note that this means that the unique highest weight vector of $M_n$ with respect to the $\mathfrak{sl}_n$ action (up to scaling) has weight $\lambda^{(n)}_1 + \dots + \lambda^{(n)}_{k_n}$.

Fix $L = \max(l(\mu_j))$. From now on we will consider only $n>|L|$, which is okay since all cofinite sets belong to $\mathcal F$. So for almost all $n$, $M_n$ as a representation of $\mathfrak{sl}_n$ is equal to $\bigoplus V(\mu_j|_{t_n})$, where $\mu_j|_{t_n}$ is a weight equal to $(\mu_j|_{t_n})_i = (\mu_j^\bullet)_i - (\mu_j^\circ)_{n+1-i}$. Since $M_n$ must have a unique highest weight vector, it follows that $\mu_j|_{t_n}$ must be one of them. But if $\mu_j|_{t_n}$ is the unique highest weight for some $n>|L|$, then it is the unique highest weight for all such $n$, since this property depends only on the regions where $\mu_j|_{t_n}$ is non-zero, which do not change. So denote by $\mu$ the bipartition corresponding to this unique highest weight. We have $\mu|_{t_n} = \lambda^{(n)}_1 + \dots + \lambda^{(n)}_{k_n}$.

For the next step note that each $\lambda^{(n)}_i$ can be represented as $\chi^{(n)}_i + d^{(n)}_i$, where $\chi^{(n)}$ is a partition with $\chi^{(n)}_n =0$. So we have $\mu|_{t_n} - \sum d_i^{(n)}= \eta^{(n)}_1+\dots + \eta^{(n)}_{k_n}$, where both the left and the right parts are partitions with $n$-th term being zero. Here wlog we assume that $\eta_i \ne 0$, since if it happens to be so, the corresponding evaluation module is trivial as a $Y(\mathfrak{sl}_n)$-module and hence we can ignore it. The left part is a partition with rows of lengths equal to lengths of the rows of $\mu^{\bullet}$ and $n-a_j$, where $a_j$ are lengths of the rows of $\mu^{\circ}$. Now each $\eta^{(n)}_{k_n}$ must be equal to the sum of several rows of this partition, hence overall there is a finite number of ways to choose  $\eta^{(n)}_j$. Thus for almost all $n$, $\eta^{(n)}_j$ correspond to the same rows of $\mu$. This means that for almost all $n$ we have $\eta^{(n)}_j= \eta_j|_{t_n} - (\eta_j|_{t_n})_n$ for a fixed bipartition $\eta_j$ corresponding to the same collection of rows of $\mu$. Hence for almost all $n$ we have:
$$
M_n \sqsubset L(\eta_1|_{t_n}+c_1^{(n)}) \otimes \dots \otimes L(\eta_k|_{t_n} + c_k^{(n)}) ,
$$
for some other $c^{(n)}_i \in \overline{\mathbb F}_{p_n}$. But  the ultraproduct of $\overline{\mathbb F}_{p_n}(c_i^{(n)})$ is equal to $\prod_{\mathcal F}\overline{\mathbb F}_{p_n}(c_i^{(n)}) = \mathbb C(c_i)$, for some $c_i \in \mathbb C$ and we can obtain $\mathbb C(c_i)$ with any $c_i$  in this way. Thus from the ultraproduct construction we have:
\begin{equation}
M \sqsubset L(\eta_1 + c_1) \otimes \dots \otimes L(\eta_k+c_k) .
\end{equation}
Note that since $\eta_1 + \dots + \eta_k = \mu$ it follows that $\sum c_k = 0$ and hence the action of $\mathfrak{gl}_t$ on the tensor product is standard.
\end{proof}

\begin{cor}
An irreducible module $M$ is completely determined by the collection of bipartitions $\lambda_i$ and complex numbers $c_i$.
\end{cor}
\begin{proof}
Indeed, if two irreducible modules through the construction above are subquotients of the same tensor product of evaluation modules, it means that the corresponding evaluation modules for finite rank Yangians are isomorphic for almost all $n$, hence the modules themselves are isomorphic.
\end{proof}

Now we can generalize the parametrization of irreducible modules by highest-weight:
\begin{def0}
Consider the realization of $M$ as a subquotient of a tensor product of evaluation modules defined in (3). Then the highest weight of $M$ is a pair of sequences of elements of $\mathbb C[[u^{-1}]]$ equal to 
$$
\lambda^\bullet_i(u) = \prod_k(1 +[(\eta^\bullet_k)_i+c_k]u^{-1}) \textrm{ and } \lambda^\circ_i(u) = \prod_k(1 +[(\eta^\circ_k)_i - c_k]u^{-1})
$$
and an element $\lambda^m = \prod_k(1 +c_ku^{-1})$, defined up to a simultaneous multiplication with any $f \in \mathbb C[[u^{-1}]]$.
Also denote $l(\lambda) = \max \ l(\eta_i)$, $l(\lambda^\bullet) = \max \  l(\eta^\bullet_i)$, $l(\lambda^\circ) = \max \ l(\eta^\circ_i)$.
\end{def0}
\begin{rem}
 Since there is a finite number of $\eta_k$ it follows that sequences $\lambda^\bullet(u)$ and $\lambda^\circ(u)$ stabilize at  infinity for any $M$. Also note that for $i > l(\lambda^\bullet)$ it holds that $\lambda^\bullet_i = \lambda^m$ and that for $i > l(\lambda^\circ)$ it holds that $\lambda_i^\circ(u)= \lambda^m(-u)$.
\end{rem}
This is well-defined since if two pairs $\lambda$ and $\mu$ are highest weights of the same module, their restriction to finite rank with $n > |\lambda|$ equals to $(\lambda|_{t_n})_i = \lambda^\bullet_i(u)$ for $i \le |\lambda^\bullet|$, $(\lambda|_{t_n})_{n-i+1} = \lambda^\circ_i(-u)$ for $i \le |\lambda^\circ|$ and $(\lambda|_{t_n})_i = \lambda^m(u)$ for the rest, (Where the coefficients are also reduced to $\overline{\mathbb F}_{p_n}$ through a fixed isomorphism of ultraproducts), are the same up to multiplication by $f$ for almost all $n$, hence the same up to multiplication by the ultraproduct of $f$'s for complex rank.

We can also generalize the classification of irreducible modules by collections of Drinfeld polynomials. To do this, we will need the following lemma:
\begin{lemma}
  Suppose  $\lambda_1 = \prod_i (u+l^1_i + c_i)$ and $\lambda_2 = \prod_i(u+l^2_i + c_i)$, where $l_i$ are integers and $c_i \in \mathbb C$, satisfy $\frac{\lambda_1(u)}{\lambda_2(u)} = \frac{P(u+1)}{P(u)}$ for some monic $P(u)$. Suppose also that $p> \deg P(u) + 1$. Then for $\mu_1(u) = \prod_i(u+l^1_i + d_i)$ and $\mu_2(u) = \prod_i(u+l^2_i+d_i)$ there exists a monic polynomial $Q(u)$ of the same degree as $P$ s.t. $\frac{\mu_1(u)}{\mu_2(u)} = \frac{Q(u+1)}{Q(u)}$.
\end{lemma}
\begin{proof}
 Let's group the roots of $P$ into a group of "strings" in the following way. Consider the divisor $D$ of zeroes of $P$. Pick a point $p_0$. Next find a point $p$ which lies in the same class in $\overline{\mathbb F}_p/\mathbb F_p$, such that it lies in $D$, but $p-1$ does not lie in $D$ (this is always possible since the number of zeroes is less than $p-1$). Now denote by $s_1$ the set $\{p,p+1,p+2,\dots,p+l\}$ for maximal $l$ such that $s_1 \subset D$. Now take $D - s_1$ and repeat. Then $D = \bigcup_i s_i$. It follows that if we denote by $r_i$ the minimal element in $s_i$ and by $t_i$ the maximal element, then $\frac{P(u+1)}{P(u)} = \prod \frac{u-s_i+1}{u-t_i}$. Now if we group $c_i$ in groups by their image in $\mathbb C/\mathbb Z$, it follows that after renumbering inside the groups (which changes nothing) we can assume that $l^1_i+c_i = -s_i+1$ and $l^2_i+c_i = -t_i$. Now if we move to $\mu$, it follows that we can just consider $Q$ defined by the divisor $D_i' = \bigcup (s_i -c_i +d_i)$, where if $s_i = \{ p_i, p_i + 1, \dots, p_i + l_i\}$, then $s_i -c_i+d_i = \{p_i-c_i+d_i, \dots, p_i-c_i+d_i+l_i\}$.
\end{proof}

We are ready to prove the main classification theorem:

\begin{thm}
For every irreducible $M\in \Rep_0(Y(\mathfrak{sl}_t))$ there exists a pair of sequences of monic polynomials in $\mathbb C[u]$, denoted by $P(u) =(P^\bullet_i(u), P^\circ_i(u))$ , such that the corresponding highest-weight satisfies:
\begin{equation}
\frac{\lambda^\bullet_i(u)}{\lambda^\bullet_{i+1}(u)} = \frac{P^\bullet_i(u+1)}{P^\bullet_i(u)} \ , \ \  \frac{\lambda^\circ_i(-u)}{\lambda_{i-1}^\circ(-u)}=\frac{P^\circ_{i-1}(u+1)}{P^\circ_{i-1}(u)} \ ,
\end{equation}
and both sequences of polynomials stabilize and equal to $1$ for sufficiently large $i$.
 Moreover, for any such $P(u)$, there is $M$ with highest weight satisfying (4). 
 This also gives a $1-1$ correspondence between irreducible modules and sequences of  polynomials $P(u)$.

\end{thm}

\begin{def0}
We will call the polynomials belonging to the sequence $P(u)$ as above Drinfeld polynomials.
\end{def0}

\begin{proof}

For the first part, consider the corresponding sequence of $M_n \in \Rep_0(Y(\mathfrak{sl}_n))$ almost all of which are irreducible with highest weight equal to $\lambda|_{t_n}(u)$, where $\lambda(u)$ is the highest weight of $M$. Now, for almost all $n$, we have $\frac{(\lambda|_{t_n})_i(u)}{(\lambda|_{t_n})_{i+1}(u)} = \frac{P^{(n)}_i(u+1)}{P^{(n)}_i(u)}$. For $n$ big enough (bigger than $l(\lambda)$) this means $\frac{(\lambda^\bullet_i)|_{t_n}(u)}{(\lambda^\bullet_{i+1})|_{t_n}(u)} = \frac{P^{(n)}_i(u+1)}{P^{(n)}_i(u)}$ and $\frac{(\lambda^\circ_{i+1})|_{t_n}(-u)}{(\lambda^\circ_{i})|_{t_n}(-u)} = \frac{P^{(n)}_{n-i}(u+1)}{P^{(n)}_{n-i}(u)}$, for some $P^{(n)}_i \in \overline{\mathbb F}_{p_n}[u]$. Also note that for big enough $n$  for $i \ge |\lambda^\bullet|$ and $i \le n - 1-|\lambda^\circ|$ we have $P^{(n)}_i=1$. From Lemma 4.2.9 it follows that all $P^{(n)}_i$ have the same degree not depending on $n$ for almost all $n$, since $\lambda|_{t_n}$ satisfy the condition.

Introduce the notation  $P^\bullet_i = \prod_{\mathcal F} P^{(n)}_i$, which are well defined since the degree is constant. Also set $P^\circ_i = \prod_{\mathcal F} P^{(n)}_{n-i}$. Now, by \lthm, it follows that these ultraproducts satisfy the required equations, since $\lambda^\bullet_i $ and $\lambda^\circ_i$ are also ultraproducts. And the equation holds for almost all $n$.

Fix a pair $P(u)$. Denote by $n_{\bullet}$ the number of polynomials in $P^\bullet$ and by $n_{\circ}$ the same for $P^\circ$. To prove the converse statement, consider the highest weight:
\begin{align}
& \mu_i^\bullet(u) = u^{-k}P^\bullet_1(u) \dots P^\bullet_{i-1}(u)P^\bullet_{i}(u+1) \dots P^\bullet_{n_\bullet}(u+1)P^\circ_{n_\circ}(u+1) \dots P^\circ_1(u+1) \\
& \mu^m (u)= u^{-k}P^\bullet_1(u)P^\bullet_{n_\bullet}(u)P^\circ_{n_\circ}(u+1) \dots P^\circ_1(u+1) \\
& \mu_i^\circ(-u) = u^{-k}P^\bullet_1(u)P^\bullet_{n_\bullet}(u)P^\circ_{n_\circ}(u) \dots P^\circ_{i}(u)P^\circ_{i-1}(u+1) \dots \dots P^\circ_1(u+1) \ .
\end{align} 
This highest weight obviously satisfies  (4). Also, the corresponding irreducible finite rank modules are finite-dimensional, and since as $\mathfrak{sl}_{t_n}$-modules they are subquotients of a fixed tensor product of simple modules, it follows that there is a finite number of possibilities for their structure as an $\mathfrak{sl}_{t_n}$-module, so for almost all $n$ they are the same, and the ultraproduct of these modules is well defined. Hence we have an irreducible $M \in \Rep_0(Y(\mathfrak{sl}_t)) $ with this highest weight. (Here we also use Lemma 4.2.5 and Thm 3.2.3)

Let us prove the last statement. Since the highest weight of $M$ is unique up to multiplication by $f$, it follows that the ratios in (4) are determined uniquely for a fixed $M$, hence the Drinfeld polynomials are determined uniquely. Indeed this is obvious for transcendental $t$ and for algebraic $t$ we just need to take the Polynomials which are not divisible by any $q_{p_n}(u+c)$ for almost all $n$, since any other choice won't lead us to a well defined ultraproduct (the degrees would increase to infinity). If two irreducible modules $M$ and $N$ have the same Drinfeld polynomials, it means that the corresponding $M_n$ and $N_n$ have the same Drinfeld polynomials and hence are isomorphic for almost all $n$, thus $M\simeq N$.
\end{proof}

To finish, let's discuss the classification of $Y(\mathfrak{gl}_t)$-modules. Since we already know that $Y(\mathfrak{gl}_t) = Y(\mathfrak{sl}_t)\otimes \mathbb C[z_1,z_2,\dots]$ by Cor. 4.1.5(c), it follows that to fix an irreducible representation of $Y(\mathfrak{gl}_t)$ it is enough to fix an irreducible representation of $Y(\mathfrak{sl}_t)$ and a series of complex numbers. Now two such representations corresponding to the same irreducible representations of $Y(\mathfrak{sl}_t)$ differ by multiplication by a one-dimensional representation of $Y(\mathfrak{gl}_t)$. The only difference is that $f(u)$ defining the structure of $Y(\mathfrak{gl}_t)$ representation on $\mathbb C$ should be of the form $f(u) = 1 + u^{-2}\mathbb C[u]$, since we want the action of $\mathfrak{gl}_t$ on our modules to be standard. So we have the following corollary: 
\begin{cor}
Irreducible objects of $\Rep_0(Y(\mathfrak{gl_t}))$ are in $1-1$ correspondence with tuples $(P(u),f(u))$, where $P(u)$ is a pair of sequences of Drinfeld polynomials and $f(u) \in 1 + u^{-2}\mathbb C[[u^{-1}]]$.
\end{cor}

\begin{rem}
 If we are interested in irreducible objects of $\Rep_0'(Y(\mathfrak{gl_t}))$, we should drop the requirement of linear term in $f(u)$ to be zero and let $f(u)$ run over $1 + u^{-1}\mathbb C[[u^{-1}]]$.
\end{rem}

\begin{rem}
 In a similar way one can define the Yangians and twisted Yangians of other classical Lie algebras\footnote{For defenitions of Yangians for over classical algebras see \cite{chari1995guide}, for definition if twisted Yangians see \cite{molev2007yangians}} in complex rank(see section 7 in \cite{etingof2016representation}). One can also extend in a similar manner  the classification theorems of irreducible finite dimensional representations of these algebras to complex rank. The new $RTT$-presentation for Yangians $Y(\mathfrak{sp}_n)$ and $Y(\mathfrak{so_n})$ appearing in \cite{guay2017representations} might be very useful for this purpose.
\end{rem}

\bibliographystyle{alpha}
\bibliography{biblio}

\end{document}